\documentclass[10pt]{article}
\usepackage{amssymb}
\usepackage{amsmath}
\usepackage{amsthm}
\usepackage[textwidth=16cm, textheight=24cm]{geometry}

\usepackage[english]{babel}
\usepackage[utf8]{inputenc}
\usepackage[T1]{fontenc}
\usepackage{listings}
\newtheorem{thm}{Theorem}
\newtheorem{cor}[thm]{Corollary}
\newtheorem{lem}[thm]{Lemma}

\def\Hom{\operatorname {Hom}}

\def\dim{\operatorname{dim}}

\def\Spec{\operatorname{Spec}}

\def\ideal{\lhd}

    \def\tensor{\otimes}
    
    \def\into{\hookrightarrow}

\begin{document}

    \title{Local finite dimensional Gorenstein $k$-algebras having Hilbert function $(1,5,5,1)$ are smoothable}
    \author{Joachim Jelisiejew\thanks{supported by the project
    ``Secant varieties, computational complexity, and toric degenerations''
    realised within the Homing Plus programme of Foundation
    for Polish Science,
    co-financed from European Union, Regional
    Development Fund.}}
    \maketitle

    \def\Hilb{\mathcal{H}ilb\,}
    \def\Hilbd{\Hilb_d\ \mathbb{P}^n}
    \def\Hilbdk{\Hilb_d\ \mathbb{P}^n_k}
    \begin{abstract}
        Let $k$ be an algebraically closed field of characteristic $0$. The
        question of irreducibility of the punctual Hilbert scheme $\Hilbdk$ and its Gorenstein locus
        for various $d$ was studied in
        \cite{CEVV8, CN9, CN10, CN11}.
        In this short paper we prove that the subschemes corresponding to the
        Gorenstein algebras having Hilbert function $(1,5,5,1)$ are smoothable i.e. lie
        in the closure of the locus of smooth subschemes. Among the Gorenstein algebras
        of length $12$ the smoothability of
        algebras having such Hilbert function seems to be the most inapproachable
        using non-direct tools e.g. structural theorems.
    \end{abstract}

    For a local and finite dimensional (abbreviated f.d.) $k$-algebra $(A, \mathfrak{m}, k)$ we define its Hilbert
    function $h_A$ as the Hilbert function of $\operatorname{gr}_\mathfrak{m}(A)$, then
    $h_A(n) = 0$ for $n \gg 0$ and the function is commonly written as the
    vector of its non-zero values.
    The \emph{socle degree} of the algebra $A$ is defined by $d := \max\{n: h_A(n) \neq 0\}$. The
    algebra $A$ is \emph{Gorenstein} iff it
    is injective as an $A$-module.
    Although Gorenstein algebras need not be graded in this paper we encounter
    only graded ones.

    \def\Hilbzero{\Hilb^\circ_d\ \mathbb{P}^n}
    \def\Hilbd{\Hilb_d\ \mathbb{P}^n}
    The Hilbert scheme $\Hilbd$ parametrizing subschemes of
    length $d$ has an irreducible component $\Hilbzero$ being the closure of smooth
    subschemes, hence the question of irreducibility is the question of
    equality. A scheme $R = \Spec A \subseteq \mathbb{P}^n$ of length $d$ is
    called \emph{smoothable} if $[R] \in \Hilbzero$. This is
    equivalent to the existence of an abstract smoothing of the algebra $A$
    i.e. a flat family of $k$-algebras with the general fiber smooth
    and the special fiber isomorphic to $A$ (see \cite{CEVV8} for details).

    Every algebra having
    Hilbert function $h_A = (1, 5, 5, 1)$ can be embedded into $\mathbb{A}_k^5
    \into \mathbb{P}_k^5$, thus we fix
    a $5$-dimensional $k$-linear space with basis
    \[V := \bigoplus_{i=1}^5 kx_i\mbox{ and a dual space with a dual basis }V^*
    := \bigoplus_{i=1}^5 ka_i,\mbox{ so that } a_i(x_j) = \delta_{i,j}.\]
    Finally we take $W := V^* \oplus kz$ obtaining an inclusion $V^* \subseteq \mathbb{P}W$.

    \def\SVprim{S^\bullet V'}
    \def\SV{S^\bullet V}
    \def\SVstar{S^\bullet V^*}
    The study of f.d.\ Gorenstein algebras is closely related to the concept of
    inverse systems (due to Macaulay) and apolar ideals briefly explained
    below.
    The natural pairing $V \tensor V^* \to k$ may be extended to
    \def\hook{\lrcorner}
    \[\hook:\SVstar \tensor \SV \to \SV.\]
    by viewing $a_i$ as a partial derivation operator $\partial/\partial x_i$
    acting on $\SV$.
    Fixing a element $F\in \SV$ (not necessarily
    homogeneous) of total degree $d$ and restricting
    $\hook$ to $\SVstar \tensor kF$ we obtain a linear map $\SVstar \to \SV$.
    Its kernel $I_F$ is an ideal of $\SVstar$ called
    the \emph{apolar ideal} of the form $F$ and the residue algebra $\SVstar / I_F$
    is called the \emph{apolar algebra} of $F$. It is a local f.d. Gorenstein $k$
    algebra with socle degree
    $d$ and all such algebras arise this way (see \cite[Chap. 21]{Eisenbud} for
    details). If $F$ was homogeneous with respect to the total degree then
    $I_F$ is homogeneous and the apolar algebra is graded.

    \newpage
    \section{Smoothability}

    The main result of the paper is
    \begin{thm}
        Local finite dimensional Gorenstein $k$-algebras having Hilbert function $(1,5,5,1)$ are
        smoothable.
        \label{mainthm}
    \end{thm}
    The proof is presented at the end of this section.
    The same result has been independently
    obtained by Cristina Bertone, Francesca Cioffi and Margherita Roggero by different means
    in \cite{BCR}.

    In the series of papers \cite{CN9, CN10, CN11} Gianfranco Casnati and Roberto Notari proved that
    the Gorenstein locus of $\Hilbd$ is irreducible for $d \leq
    11,\ n\geq 1$. The same authors communicated to us
    that using Theorem \ref{mainthm} they proved
%    that they can prove that the
%    Gorenstein locus in $\Hilb_{12}(\mathbb{P}^n)$ has at most two
%    components: the smoothable one and, possibly, the ``$(1,5,5,1)$ component'',
%    whose general member represents a local algebra having Hilbert function
%    $(1, 5, 5, 1)$. Combined with our result
%    this gives:
    \begin{thm}[G. Casnati, R. Notari --- in preparation]
        \label{CN}
        For $d \leq 12$ and $n\geq 1$ the locus of $\Hilbd$ consisting of Gorenstein schemes having length $d$ is irreducible.
    \end{thm}

    The zero-dimensional Gorenstein schemes and their smoothability is also
    investigated in relation with the study of secant varieties to Veronese
    reembeddings. As a consequence of the Theorem \ref{CN} we obtain
    \begin{thm}
        Let $r, n, d, i$ be integers and let $W \simeq \mathbb{C}^{n+1}$ be a
        vector space. Let $\sigma_r(v_d(\mathbb{P}W))$ be the \mbox{$r$-th} secant
        variety of the $d$-th Veronese embedding of $\mathbb{P}W$. If $d \geq
        2r$, $r \leq i \leq d - r$ and $r \leq 12$ then
        $\sigma_r(v_d(\mathbb{P}W))$ is set-theoretically defined by $(r+1)
        \times (r+1)$ minors of the $i$-th catalecticant matrix.
    \end{thm}
    See \cite[\S1.1]{BuBu}  for explanation of notation used in the statement
    and \cite[Thm 1.6, Cor. 1.11]{BuBu} for the resulting of this theorem from Theorem \ref{CN}.
    In this language Theorem \ref{mainthm} may be restated as: any form of
    degree $d\geq 3$ in
    $V \oplus \mathbb{C}z$, which
    can be written as $f\cdot z^{d-3}$ where $f\in S^3 V$, belongs to the $12$-th
    secant variety of $d$-th Veronese embedding of $\mathbb{P}(V \oplus
    \mathbb{C}z)$ (see \cite[\S8.1]{BuBu}).

%    \smallskip
%    Essentially the proof of Theorem \ref{mainthm} consists of choosing an
%    algebra $A$ with Hilbert function $(1, 5, 5, 1)$ a dimension of tangent space at $[A]$

\def\Hilbzerofive{\Hilb^\circ_{12}\ \mathbb{P}W}
    \smallskip
%    The idea of the proof of Theorem \ref{mainthm} is to choose an apolar
%    algebra of the for
%    $A$ having Hilbert function $(1, 5, 5, 1)$,
%    prove that it is smoothable and check that the Zariski tangent space at
%    the point $[\Spec A]\in \Hilbzerofive$ has dimension
%    $60 = 12\cdot 5$.
    The idea of the proof of Theorem \ref{mainthm} is to choose an appropriate
    apolar algebra $A = \SVstar/I_F$ and prove that it has Hilbert function $(1, 5, 5, 1)$,
    it is smoothable and the Zariski tangent space at
    the point $[\Spec A]\in \Hilbzerofive$ has dimension
    $60 = 12\cdot 5$.

    \begin{lem}
        The ideal $I_F$ apolar to the form

        \begin{equation}\label{form}
        F := x_2^2  \cdot  x_5 + x_2  \cdot  x_4^2 + x_1^2  \cdot  x_5 +
        x_3^2  \cdot  x_4\in S^3 V
    \end{equation}
        is generated by
        \[
        a_1 \cdot a_2,\ a_1 \cdot a_3,\ a_1 \cdot a_4,\ a_2^2 - a_1^2,\ a_2 \cdot a_3,\ a_2 \cdot a_4 -
        a_3^2,\ a_2 \cdot a_5 - a_4^2,\ a_3 \cdot a_5,\ a_4 \cdot
        a_5,\ a_5^2
        \]
        and its residue algebra $A$ has the Hilbert function $h_A = (1,5,5,1)$.
        \label{lemideal}
    \end{lem}
    \begin{proof}
        By computing the partial derivatives of $F$ we conclude that no linear
        form annihilates $F$, thus $h_A(1) = 5$. The Hilbert function $h$ of a
        local f.d.\ graded Gorenstein algebra having socle degree
        $d$ satisfies $h(d-n) = h(n)$ where $0\leq n\leq d$, thus $h_A =
        (1, 5, 5, 1)$.

        A~straightforward check shows that the given elements annihilate $F$.
        Denote by $I$ the ideal generated by these elements.
        A direct check or a~computation (see source
        \texttt{quotientDimension}) shows that the residue algebra of $I$ has
        dimension $12$ as a $k$-vector space so $I$ is the apolar
        ideal of $F$.
    \end{proof}

    \begin{lem}
        Let $A_0$ be a~finitely dimensional $k$-algebra and $A := A_0[x]/Jx$,
        where $J\ideal A_0$, be a~graded algebra.

        Let $f: = m - x^2$, where $m\in J$, then
        there exists a flat family over $k[t]$ with special fiber isomorphic
        to $A/f$ and a general fiber of the form $A_0[x]/I_1\cap I_2$, where
        $I_1 + I_2 = 1$, $A_0[x]/I_1  \simeq A_0/mJ$ and $A_0[x]/I_2  \simeq A_0/J$.
    \end{lem}

    \begin{proof}
        Put $f_\alpha := m - \alpha x - x^2$ for any $\alpha\in k$.
        Consider $A[t]/(m - t\cdot x - x^2)$ over $k[t]$. The fiber over $t =
        0$ is the algebra $A/f$. Take $\alpha\in k\setminus\{0\}$. Since $(x^2,
        x+\alpha) = (1)$ and $x(m - \alpha x -x^2) = -\alpha x^2 - x^3$ the following equality of
        $A$-ideals holds
        \[
        (f_{\alpha}, x^2) \cap (f_{\alpha}, x + \alpha) = (f_{\alpha}, x^2)\cdot (f_{\alpha}, x+\alpha) = (f_{\alpha}, x^3
        + \alpha x^2) = (f_{\alpha}).
        \]
        Now $(f_{\alpha}, x^2) = (m-\alpha x, x^2) = (m - \alpha x)$ so that $A/(f_{\alpha},
        x^2)  \simeq A_0/mJ$ and $(f_{\alpha}, x + \alpha) = (m, x + \alpha)$, so that $A/(f_{\alpha}, x +
        \alpha)  \simeq A/(m, x + \alpha)  \simeq A_0/J$.

        Now $\dim_k A/f =  \dim_k A_0/mJ + \dim_k
        A_0/J$, thus
        the length of fibers is constant. The algebra $A[t]/(m - tx - x^2)$ is finite
        over $k[t]$ so by \cite[Thm III.9.9]{Har}
        the family $k[t] \to A[t]/(m - tx - x^2)$ is flat, which proves the
        claim.
    \end{proof}

    \newpage
    \begin{cor}
        \label{smoothable}
        The apolar algebra of the form $F$ defined in Lemma \ref{form} is smoothable.
    \end{cor}
    \begin{proof}
        \def\SVdualprim{S^\bullet V'^*}
        \def\SVdual{S^\bullet V^*}
        Let $V' \subseteq V$ be spanned by $a_2, a_3, a_4, a_5$.

        It follows from Lemma \ref{lemideal} that the apolar ideal $I_F$ of $F$ has the form $J_0 +
        J_1\cdot (a_1) + (a_2^2 - a_1^2)$ where $J_0, J_1 \ideal \SVdualprim$,
        $J_0 \subseteq J_1 = (a_2, a_3, a_4),\ a_2^2\in J_1$ and
         $a_2^2\cdot J_1 \subseteq J_0$.

        We identify $\SVdual$ with the polynomial algebra $\SVdualprim[a_1]$.
        Applying the previous lemma to the algebra $A_0 = \SVdualprim/J_0$,
        ideal $J := J_1$ and element $m := a_2^2$ we see that the apolar
        algebra $\SVdual/I_F$ is a~flat degeneration of the algebras
        of the form $\SVdual/I_1 \cap I_2$, where $I_1 + I_2 = 1$,
        $\SVdual/I_1  \simeq  \SVdualprim/J_0 +a_2^2\cdot J_1 = \SVdualprim/J_0$ and $\SVdual/I_2  \simeq
        \SVdualprim/J_1$.

        The quotient $\SVdualprim/J_0$ is canonically isomorphic to the apolar algebra to
        the form $F(x_1=0) = x_2^2  \cdot  x_5 + x_2  \cdot  x_4^2  + x_3^2
        \cdot  x_4$. One can check, as in proof of Lemma \ref{form}, that the
        Hilbert function of this algebra is $(1, 4, 4, 1)$ so it is Gorenstein
        and thus smoothable (see
        \cite{CN10}).
        Now a~dimension count shows that $\SVdual/I_2$ is a~$2$-dimensional
        algebra over $k$, so it is smoothable as well.

        Finally smoothability can be checked on irreducible components
        (\cite[\S4.1]{CEVV8ar}), so the smoothability of both $\SVdual/I_1$
        and $\SVdual/I_2$ implies
        the smoothability of $\SVdual/I_1 \cap I_2$ proving that $\SVdual/I_F$ is
        a~flat degeneration of smoothable algebras and thus it is smoothable.
    \end{proof}

    \begin{lem}
        \label{tangent}
        The tangent space to the Hilbert scheme of $V^*$ at point $[R]$
        corresponding to the scheme $R$,
        where $I_F$ is the apolar ideal to $F$
        defined in \ref{form} and $R := \Spec \SV^*/I_F$, has dimension $60$.
    \end{lem}
    \begin{proof}
        The inclusion $V^* \subseteq \mathbb{P}W$ induces an inclusion of
        Hilbert schemes. The tangent space to a~finite projective scheme
        defined by the saturated homogeneous ideal $I$ in the total coordinate
        ring $S = k[W]$ is canonically identified with $\dim_k \Hom_S(I, S/I)$
        (see \cite[Thm 1.1b]{HarDef}).

        View $R$ as the projective scheme defined by the homogenization
        of the ideal $I$, which is easily seen to be saturated.
        We can identify $\Hom_S(I, S/I) = \Hom_S(I/I^2, S/I) =
        \Hom_{S/I}(I/I^2, S/I)$. Now $S/I$ is Gorenstein, thus self injective,
        so $\Hom(-, S/I)$ is exact and by looking on small extensions we see
        that $\dim_k
        \Hom_{S/I}(I/I^2, S/I) = \dim_k
        I/I^2 = \dim_k S/I^2 - \dim_k S/I$. A computer algebra check shows that this dimension is equal
        to $60$ (see source \texttt{quotientDimension}).
    \end{proof}

    \def\Hilbtwelve{\Hilb_{12}\ \mathbb{P}W}
    \def\Hilbtwelvezero{\Hilb_{12}^\circ\ \mathbb{P}W}
%    To deduce smoothability of every algebra from the smoothability of one
%    we need to know that the set of algebras having Hilbert function
%    $(1, 5, 5, 1)$ is irreducible in $\Hilbtwelve$. For that we use a deep
%    result by Iarrobino, check \cite[Theorem]{iarrobino}, whose corollary is

    \subsection*{Proof of Theorem \ref{mainthm}}

%    \begin{prop}
%        Local finite dimensional \emph{Gorenstein} $k$-algebras having Hilbert function $(1,5,5,1)$ are
%        smoothable.
%    \end{prop}
    \begin{proof}
        \def\Zlocus{\mathcal{G}}
        Recall that we have fixed an inclusion $V^* \subseteq \mathbb{P}W$.
        A local f.d.\ Gorenstein algebra having Hilbert function $(1, 5, 5, 1)$ is isomorphic to the algebra of global sections of
        a~finite closed Gorenstein subscheme of $\mathbb{P}W$.
        The locus $\Zlocus$ of points corresponding to such subschemes is an irreducible
        subset of $\Hilbtwelve$ by \cite[Thm I]{iarrobino}, it sufficies to
        prove that is it contained in $\Hilbtwelvezero$.

        Suppose on contrary that the locus $\Zlocus$ is
        contained in an irreducible component of $\Hilbtwelve$ other
        than the closure of smooth schemes.
        The scheme $R$ defined in Lemma \ref{tangent} is smoothable by Cor. \ref{smoothable}, so $[R]$ lies in
        the intersection of these two components, thus the
        dimension of the Zariski tangent space at $[R]$ is greater than the
        dimension of $\Hilbtwelve$, which is $60$. This
        contradicts Lemma \ref{tangent}.
    \end{proof}

    \section{Magma source code}
    The following {\cite{Magma}} source code may be helpful in proving Lemmas
    \ref{form} and \ref{tangent}. In fact earlier versions of the proof used much
    more of the Magma computational capacity.
    Magma console is available for use on-line at \texttt{http://magma.maths.usyd.edu.au/calc/}.

%    \lstset{language=Magma, basicstyle=\ttfamily, keywordstyle=\bfseries}
    \lstset{basicstyle=\ttfamily}
    \begin{lstlisting}
        quotientDimension := function(ideal_)
            return Dimension(quo<Generic(ideal_) | ideal_>);
        end function;

        Q := RationalField();
        P<a1, a2, a3, a4, a5> := PolynomialRing(Q, 5);
        I := Ideal([a1*a2, a1*a3, a1*a4, a2^2 - a1^2, a2*a3,
            a2*a4 - a3^2, a2*a5 - a4^2, a3*a5, a4*a5, a5^2]);

        quotientDimension(I);
        quotientDimension(I^2) - quotientDimension(I);
    \end{lstlisting}

    \end{document}